\pdfoutput=1
\documentclass[12pt,a4paper]{amsart}
\usepackage[top=1.15in, bottom=1.15in, left=1.17in, right=1.17in]{geometry}
\usepackage{amsthm, amsmath, amssymb, amscd, latexsym, multicol, verbatim, enumerate, graphicx,xy, color, stmaryrd}
\usepackage{tikz}
\usepackage{float}
\usepackage{hyperref}
\hypersetup{colorlinks=true, pdfstartview=FitV, linkcolor=blue, citecolor=blue, urlcolor=blue, breaklinks=true}
\usepackage[english]{babel}
\usepackage[T1]{fontenc}
\usepackage[utf8]{inputenc}
\usepackage{mathrsfs}
\usepackage{ytableau}
\usepackage{cleveref}
\usepackage{tikz-cd}
\usepackage{filecontents}
\usepackage{adjustbox}
\usetikzlibrary{arrows,backgrounds,calc,trees,hobby}
\usepackage{csquotes}
\usepackage{graphicx}
\usepackage{tabularx,booktabs}
\usepackage[clockwise]{rotating}
\usepackage{xspace}
\usepackage[ruled, vlined, linesnumbered]{algorithm2e}

\renewcommand{\prec}{\lessdot}

\usepackage[backend=bibtex,style=alphabetic,doi=false,isbn=false,url=false,giveninits=true,maxbibnames=50]{biblatex}
\newcounter{cnt}
 \makeatletter
\def\mydggeometry{\makeatletter\dg@YGRID=1\dg@XGRID=20\unitlength=0.003pt\makeatother}
\makeatother

\newtheorem{theorem}{Theorem}[section]
\newtheorem*{theorem*}{Theorem}
\newtheorem{lemma}[theorem]{Lemma}
\newtheorem{corollary}[theorem]{Corollary}
\newtheorem{proposition}[theorem]{Proposition}

\newtheorem*{conjecture*}{Conjecture}

\theoremstyle{definition}
\newtheorem{definition}[theorem]{Definition}

\newcommand{\n}{\mathfrak{n}}

\newcommand{\OP}{\mathcal{O}}
\newcommand{\CP}{\mathcal{C}}

\newcommand{\Z}{\mathbb{Z}}

\newcommand{\R}{\mathbb{R}}

\newcommand{\comp}{{\operatorname*{Comp}}}

\DeclareMathOperator{\codim}{codim}
\newcommand{\Ptau}{P_\tau}

\pgfdeclarelayer{background}
\pgfsetlayers{background,main}

\newcommand{\convexpath}[2]{
[   
    create hullnodes/.code={
        \global\edef\namelist{#1}
        \foreach [count=\counter] \nodename in \namelist {
            \global\edef\numberofnodes{\counter}
            \node at (\nodename) [draw=none,name=hullnode\counter] {};
        }
        \node at (hullnode\numberofnodes) [name=hullnode0,draw=none] {};
        \pgfmathtruncatemacro\lastnumber{\numberofnodes+1}
        \node at (hullnode1) [name=hullnode\lastnumber,draw=none] {};
    },
    create hullnodes
]
($(hullnode1)!#2!-90:(hullnode0)$)
\foreach [
    evaluate=\currentnode as \previousnode using \currentnode-1,
    evaluate=\currentnode as \nextnode using \currentnode+1
    ] \currentnode in {1,...,\numberofnodes} {
  let
    \p1 = ($(hullnode\currentnode)!#2!-90:(hullnode\previousnode)$),
    \p2 = ($(hullnode\currentnode)!#2!90:(hullnode\nextnode)$),
    \p3 = ($(\p1) - (hullnode\currentnode)$),
    \n1 = {atan2(\y3,\x3)},
    \p4 = ($(\p2) - (hullnode\currentnode)$),
    \n2 = {atan2(\y4,\x4)},
    \n{delta} = {-Mod(\n1-\n2,360)}
  in 
    {-- (\p1) arc[start angle=\n1, delta angle=\n{delta}, radius=#2] -- (\p2)}
}
-- cycle
}

\newcommand\polymake{\texttt{polymake}\xspace}

\usepackage{enumitem}
\setlist[enumerate,1]{label={\alph*.}}
\setlist[enumerate,2]{label={(\roman*)}}

\begin{document}
\author[Ahmad, Fourier \& Joswig]{Ibrahim Ahmad \and Ghislain Fourier \and Michael Joswig}
\title[Order and chain polytopes of maximal ranked posets]{Order and chain polytopes of \\ maximal ranked posets}

\address{Chair of Algebra and Representation Theory, RWTH Aachen University, Pontdriesch 10-16, 52062 Aachen, Germany}
\address{Technische Universit\"at Berlin, Chair of Discrete Mathematics / Geometry, Berlin, \& MPI Mathematics in the Sciences, Leipzig, Germany} 
\email{ahmad@art.rwth-aachen.de}
\email{fourier@art.rwth-aachen.de}
\email{joswig@math.tu-berlin.de}

\maketitle

\begin{abstract}
  The order and chain polytopes, introduced by Richard P. Stanley (1986), form a pair of Ehrhart equivalent polytopes associated to a given finite poset.
  A conjecture by Takayuki Hibi and Nan Li from 2016 states that the $f$-vector of the chain polytope dominates the $f$-vector of the order polytope. 
  In this paper we prove a stronger form of that conjecture for a special class of posets.
  More precisely, we show that the $f$-vectors increase monotonically over an admissible family of chain-order polytopes for such posets.
\end{abstract}

\section{Introduction} 
For a given finite poset $P$ with $n$ elements, Stanley \cite{Sta86} introduced the \emph{order polytope}
\begin{equation}\label{eq:order-polytope}
  \OP(P) = \{ x \in \R^n_{\geq 0} \mid x_p \leq 1 \text{ and } x_p \leq x_q \text{ if } p \prec q \}
\end{equation}
and the \emph{chain polytope}
\begin{equation}\label{eq:chain-polytope}
  \CP(P) = \{ x \in \R^n_{\geq 0} \mid x_{p_1} + \dots + x_{p_s} \leq 1 \text{ for all } p_1 \prec p_2 \prec \dots \prec p_s\} \enspace ,
\end{equation}
where $\prec$ denotes the covering relation in $P$.
Both polytopes have full dimension $n$.
The primary result of \cite{Sta86} establishes that these two polytopes are Ehrhart-equivalent, meaning they share identical Ehrhart polynomials.
Specifically, the lattice point counts of $\OP(P)$ and $\CP(P)$ are equal.
In this work, we focus on exploring the combinatorial properties of order and chain polytopes.

For an $n$-dimensional polytope, the \emph{$f$-vector} $(f_0, f_1, \dots, f_{n-1})$ counts the faces per dimension.
Stanley \cite[Theorem 3.2]{Sta86} showed that $f_0(\OP(P))=f_0(\CP(P))$.
Hibi and Li \cite{HL16} proved $f_{n-1}(\OP(P))\leq f_{n-1}(\CP(P))$.
Moreover, they proved that $\OP(P)$ and $\CP(P)$ are affinely isomorphic if and only if the poset $P$ does not contain five pairwise distinct elements $a,b,c,d,e$ such that $d,e \prec c \prec a,b$.
Subsequently, Hibi et al.\ \cite{HLSS17} established $f_1(\OP(P)) \leq f_1(\CP(P))$.
These results lead to the natural conjecture (the \emph{Hibi--Li Conjecture}) that the $f$-vector of the chain polytopes always dominates the $f$-vector of the order polytope \cite[Conjecture 2.4]{HL16}.
Recently, this conjecture was proved by Freij-Hollanti and Lundström in a special case \cite{freijhollanti2024fvector}.

A particularly promising approach to addressing the Hibi–Li Conjecture arises from the following construction.
Given a poset $P$, we consider a partition $P = C \sqcup O$ with $\hat{0} \in C$ and $\hat{1} \in O$.
This partition defines the \emph{chain-order polytope} $\OP_{C,O}$, introduced in \cite{FF16-Mark} and later generalized in \cite{FFLP, FFP20}.
When the partition of $P$ into $C$ and $O$ is sufficiently well-structured, the chain-order polytopes $\OP_{C,O}(P)$ remain Ehrhart-equivalent to $\OP(P)$ and $\CP(P)$ \cite[Theorem 3.1]{FF16-Mark}.
By appropriately varying the partition, one obtains an entire family of polytopes that \enquote{interpolate} between the order and chain polytopes of the poset.
A stronger version of the Hibi–Li Conjecture posits that the $f$-vectors of these polytopes increase monotonically throughout this family \cite{FF16-Mark}.

Our main result (Theorem~\ref{Hibi-Li-Claim-Crucial}) confirms the strong Hibi--Li Conjecture for a certain class of posets.
A poset is \emph{ranked} if each maximal chain has the same length.
For a vector $\tau=(\tau_1,\tau_2,\dots,\tau_\ell)$ of positive integers, we define the ranked poset $P_\tau$ which has $\tau_i$ elements in rank $i$, and any two elements in distinct ranks are comparable.
All ranked posets with $\tau_i$ elements in rank $i$ are subposets of $P_\tau$.
Therefore, we call $P_\tau$ a \emph{maximal ranked poset}.
Fixing a parameter $k$ with $0\leq k\leq\ell$, we define $C$ as the set of elements of rank at most $k$ and $O$ as the elements of larger rank.
Then $\OP_{C,O}(P_\tau)=\OP(P_\tau)$ for $k=0$,  and  $\OP_{C,O}(P_\tau)=\CP(P_\tau)$ for $k=\ell$.
We show that, for any maximal ranked poset $P_\tau$, the $f$-vectors of the chain-order polytopes of $P_\tau$ monotonically increase with $k$.
In particular, we obtain an independent new proof of the Hibi--Li conjecture for maximal ranked posets.
The latter form a subclass of the posets studied by Freij-Hollanti and Lundström \cite{freijhollanti2024fvector}.
Our proof exploits combinatorial properties of posets to arrive at normal forms for faces of the chain-order polytopes.
In \cite[p.~12]{Sta86}, Stanley pointed out that the faces of the chain polytopes are difficult to describe, whereas face structure of the order polytopes is known.
We hope that our proof serves as an inspiration for understanding the face structure of more general classes of chain polytopes.

The paper is organized as follows.
In Section~\ref{sec-prel}, we introduce the objects of interest.
Section~\ref{sec-normal} is devoted to a normal form for the faces of chain-order polytopes.
The proof of the main result is contained in Section~\ref{sec-hibili}.
Algorithmic aspects and computational experiments are the focus of Section~\ref{sec-comp}.
Here combinatorial optimization enters the picture.
In fact, computing maximal cliques in the comparability graph of a poset yields double descriptions of order and chain polytopes.

\noindent
\textbf{Acknowledgment:}
The work of all authors is funded by the Deutsche Forschungsgemeinschaft (DFG, German Research Foundation) through \textit{Symbolic Tools in Mathematics and their Application} (TRR 195, project-ID 286237555).
GF was further supported by \textit{Sparsity and Singular Structures} (SFB 1481, project-ID 442047500).
MJ was further supported by The Berlin Mathematics Research Center MATH$^+$ (EXC-2046/1, project ID 390685689).

\section{Preliminaries}\label{sec-prel}
We briefly recall known facts on the order and chain polytopes.
Throughout the entire article, we solely consider finite posets. 
Let $P$ be such a poset (with $n$ elements).
The face lattice of $\OP(P)$ can be described in terms of partitions.
So consider a partition
\begin{equation*}
    \pi=\{T_i\mid 1\leq i\leq s\}
\end{equation*}
of $P$.
We say that $\pi$ is \textit{$P$-compatible} if the relation $\leq$ on $\pi$ defined as the transitive closure of
\begin{equation*}
    B\leq C\,\text{ if and only if }p\leq q\, \text{ for some $p\in B$ and $q\in C$}
\end{equation*}
is antisymmetric. 
In this case, $\leq$ is a partial order on $\pi$. 
Further, $\pi$ is \textit{connected} if its blocks are connected as induced subposets.
A $P$-compatible partition $\pi$ of $P$ gives rise to a \emph{quotient poset} $(P/\pi,\leq)$ where $(P/\pi,\leq)$ is the poset of blocks in $\pi$.
Furthermore, the quotient map
\[
  q:(P,\leq)\rightarrow (P/\pi,\leq) \,,\ x\mapsto B\text{, for $x\in B$}
\]
is order-preserving; i.e., whenever $x\leq y$ in $P$, we have $q(x)\leq q(y)$ in $P/\pi$. 

Equipped with these definitions, we can now state the combinatorial description of the order polytope.
To keep the notation terse, we extend $P$ by a minimal and a maximal element $\hat{P} = P \cup \{ \hat{0}, \hat{1}\}$.
The first result in this direction goes back to Geissinger \cite{Geissinger} and Stanley \cite{Sta86}; see also \cite{Peg18}.
\begin{theorem}\label{Theorem-Order-Face-Partition}
Let $P$ be a poset. 
Then every non-empty face $F\subseteq\OP(P)$ can be identified with a partition $\pi_F$ of $\hat{P}$ that is connected, $\hat{P}$-compatible such that $\hat{0}$ and $\hat{1}$ are in different blocks.
Further, the map
\[
  \OP(\hat{P}/\pi_F)\rightarrow F \,,\ (x_{B})_{B\in (\hat{P}/\pi_{F})}\mapsto(x_{q(p)})_{p\in \hat{P}}
\]
is an affine isomorphism.
In particular, we have $\dim(F)=|\pi_F|-2$.
\end{theorem}

Next we recall the chain-order polytopes, first introduced in \cite{FF16-Mark} and then generalized in \cite{FFLP, FFP20}.
They provide an Ehrhart-equivalent family of polytopes that interpolate between the order and chain polytope of a given poset.
This family will be crucial for our main result.
For a poset $P$ consider a partition $P=C\sqcup O$ such that $\hat{0}\in C$ and $\hat{1}\in O$.
We are going to refer to the elements of $C$ as the \textit{chain elements} and and the elements of $O$ as \textit{order elements}.
\begin{definition}\label{def:chain-order-polytope}
The \emph{chain-order polytope} $\OP_{C,O}(P)\subseteq\R^{n}$ is given by
\begin{enumerate}
%
\item{$x_{\hat{0}}=0$ and $x_{\hat{1}}=1$,}
\item{for each covering relation $p \prec q \in  O$ an inequality $x_p \leq x_q$,}
\item{for each $p \in  C$  an inequality $0\leq x_p$,}
\item{for each maximal chain $p_1 \prec  p_2 \prec \dots \prec  p_r \prec a$ with $p_i \in  C$, $a\in O$ and $r \geq  0$ an inequality
\begin{equation*}
  x_{p_{1}} +\dots +x_{p_{r}} \leq x_a \enspace.
\end{equation*}}
\end{enumerate}
\end{definition}

It is known that the inequalities (\ref{eq:order-polytope}) and (\ref{eq:chain-polytope}) of the order and chain polytopes, respectively, are facet defining.
In contrast the facet description of the chain-order polytopes are unknown in general; see \cite[Conjecture 8.2]{FFLP}.
However, in our scenario, where $P$ is ranked, the facets are known.
\begin{proposition}[{\cite[Proposition 8.6]{FFLP}}]\label{prop-tame-facet}
  Let $P$ a ranked poset with partition $P=C\sqcup O$.
  Then the inequality description for $\OP_{C,O}(P)$ in Definition~\ref{def:chain-order-polytope} is facet defining.
\end{proposition}
The comparability graph of a poset $P$ is the undirected graph $\comp(P)=(P,E)$ with $\{p,q\}\in E$ if and only if $p$ and $q$ are comparable.
It follows from the description (\ref{eq:chain-polytope}) that the chain polytope $\CP(P)$ only depends on the graph $\comp(P)$.

\section{Face normal form for chain-order polytope}\label{sec-normal}
Each $k$-dimensional face of a $d$-dimensional polytope can be written as the intersection of $d-k$ facets.
However, in general, the choice of facets is not unique.
The purpose of a face normal form is to make such a selection.

From now on, we consider a maximal ranked poset $P_\tau$ where $\tau=(\tau_1,\tau_2,\dots,\tau_\ell)$ is a sequence of positive integers; here $\tau_i$ dictates the number of elements in rank $i$.
We abbreviate the notation for the order and chain polytopes by $\OP(\tau)$ and $\CP(\tau)$, respectively.
To describe the intermediate chain-order polytopes, we introduce the following concept.
\begin{definition}
  Let $\tau\in\Z_{> 0}^\ell$ be a tuple and $0\leq k\leq \ell$.
  Then we define \textit{$k$-decomposition} of $P_{\tau}$ as the decomposition $P_\tau=C\sqcup O$ with
  \begin{equation*}
    C=\bigcup\limits_{i=0}^{k}Y^{{i}},\hspace{1em}O=\bigcup\limits_{i=k+1}^{\ell+1}Y^{{i}},
  \end{equation*}
where $Y^{i}$ refers to the elements of rank $i$ and $Y^{0}=\{\hat{0}\}$ and $Y^{\ell+1}=\{\hat{1}\}$.
The decompositions for $k=0$ and $k=\ell$ are called \textit{trivial}, and the other $k$-decompositions are \textit{non-trivial}.
Throughout, we abbreviate $\OP_{C,O}(P_{\tau})$ as $\OP_{C,O}(\tau)$.
\end{definition}

To describe the face lattice of $\OP_{C,O}(\tau)$, we rephrase the facet description from Definition \ref{def:chain-order-polytope} for our class of posets $P_\tau$.
\begin{lemma}\label{Lemma-Chain-Order-Facet-Alternate}
    Let $\tau\in\Z_{> 0}^\ell$ and let $P_\tau=C\sqcup O$ be a $k$-decomposition. 
    Then the chain-order polytope $\OP_{C,O}(\tau)$is defined by the following facet-defining inequalities:
    \begin{enumerate}
        \item{$x_{\hat{0}}=0$ and $x_{\hat{1}}=1$};
        \item{for each $p \in  C$ the inequality $0\leq x_p$};
        \item{for each $a,b\in O$ with $a\prec b$ the \emph{order inequality} $x_a\leq x_b$};
        %
        \item{for any $p_i\in Y^{i}$ and $q_{k+1}\in Y^{k+1}$ the \emph{chain inequality} $x_{p_{1}} +\dots +x_{p_{k}} \leq x_{q_{k+1}}$\enspace .}
    \end{enumerate}
\end{lemma}
\begin{proof}
    The first two types are the same as in Definition \ref{def:chain-order-polytope}.
    The last two types of inequalities in the lemma together form the third class of inequalities in the original definition.
    The order inequalities arise when $r=0$, whereas the chain inequalities appear when $r>0$. The latter occurs only if the minimum of the corresponding chain is exactly $\hat{0}$.
    Proposition \ref{prop-tame-facet} confirms that these inequalities are facet-defining.
\end{proof}

We will construct a normal form for the faces of $\OP_{C,O}(\tau)$ in several steps:
Proposition \ref{Proposition-Face-Partition} will deal with faces contained in order facets but not in chain facets.
In Propositions \ref{Proposition-Zero-Variable-Face} and \ref{Proposition-Chains-Codimension}, we will consider faces that are contained in chain facets but not in order facets.
The remaining case, where order and chain facets may intersect nontrivially, will be discussed in Theorem \ref{Theorem-ChainOrder-Normal}.

\begin{proposition}\label{Proposition-Face-Partition}
Consider a $k$-decomposition of $\Ptau$.
Let $F\subseteq \OP_{C,O}(\tau)$ be a face defined only by order inequalities. 
%
Then $F$ is given by a face partition of $O$.
\end{proposition}
\begin{proof}
We are going to project downwards onto the order polytope $\OP(\Tilde{P})$,  where $\Tilde{P}$ is the induced subposet
%
\begin{equation*}
    \Tilde{P}=\{\hat{0}\}\cup O.
\end{equation*}
The latter is the poset restricted to the order part.
Now consider the projection
\[
  \phi:\OP_{C,O}(\tau)\rightarrow \OP(\Tilde{P}) \,,\ (x_p)_{p\in P_\tau}\mapsto (x_p)_{p\in O} \enspace .
\]

The image $\phi(F)$ is again a face of $\OP(\Tilde{P})$, and the equations satisfied are precisely the same as those for $F$. 
As $F$ does not satisfy any chain inequality, it follows that is does not satisfy $x_q=0$ for any $q\in O$.
This implies that $\phi(F)$ only satisfies equalities of the form $x_p=x_q$ with $p,q\in O$ for $p\leq q$, or $x_p=1$ for some $p\in Y^{{n}}$. 
Theorem \ref{Theorem-Order-Face-Partition} implies that $\phi(F)$ is given by a face partition of $\Tilde{P}$. 
However,  since not all $x\in\phi(F)$ satisfy $x_p=0$ for any $p\in O$, the element $\hat{0}$ is in a singleton block.
Thus, the face $F$ defined only by order inequalities, is entirely determined by a face partition of $O$.
\end{proof}

\begin{proposition}\label{prop-normal-form-order}
Consider a $k$-decomposition of $\Ptau$ and a face $F\subseteq\OP_{C,O}(\tau)$ given only by order inequalities and let $\varphi$ be the face partition of $O$ and let $\pi$ be the extended partition from $\varphi$ by adding singletons.
Then we have
\begin{equation*}
    \OP_{C,\varphi}(P_\tau/\pi)\cong F \enspace .
\end{equation*}
In particular, the codimension of $F$ is exactly $|O|-|\varphi|$.
\end{proposition}
Here \enquote{$\cong$} refers to affine isomorphy.
\begin{proof}
Checking that $\pi$ is a face partition will be left out as it is a straight-forward verification.
Then according to Theorem \ref{Theorem-Order-Face-Partition}, using the projection onto the blocks $q:P_\tau\rightarrow P_\tau/\pi$, we get the affine isomorphism
\[
   \OP_{C,\varphi}(P_\tau/\pi)\rightarrow F \,,\ (x_{p'})_{p'\in (\Ptau/\pi)}\mapsto (x_{q(p)})_{p\in \Ptau} \enspace .
\]

The dimension of $F$ is exactly $|\pi|-2$, where we account for the blocks containing $\hat{0}$ and $\hat{1}$.
%
However, $\hat{0}$ and all other elements of $C$ are in singleton blocks.
Hence
%
\begin{equation*}
    |\pi|-2=|C|+|\varphi|-2
\end{equation*}
%
where we account for the block in $\varphi$ containing $\hat{1}$ and $\hat{0}\in C$.
Since the chain-order polytope is of dimension $|O|+|C|-2$, the codimension equals
\begin{equation*}
    |O|+|C|-2-(|C|+|\varphi|-2)=|O|-|\varphi| \enspace . \qedhere
\end{equation*}
\end{proof}
Finally, we need to understand what the quotient poset $\Ptau/\pi$ looks like.
\begin{lemma}\label{Lemma-Quotient-Tuple}
    Fix a $k$-decomposition of $\Ptau$ for some $\tau\in\Z_{> 0}^\ell$. Let $F\subseteq\OP_{C,O}(\tau)$ be a face given only by order inequalities, and let $\pi$ be the extended partition on $\Ptau$.
    Then the quotient poset $\Ptau/\pi$ is of the form $P_{\tau'}$ for some other tuple $\tau'$.
\end{lemma}
\begin{proof}
    Instead of inspecting the entire partition, we will only consider what happens when turning an edge of the Hasse diagram into an equality. 
    The proof then follows by induction.
    Consider the edge $\{a,b\}$ for $a\in Y^{{i}}$, $b\in Y^{{i+1}}$ and $i< \ell$. Then for all $x\in Y^{{t}}$ with $t<i+1$ and $x\neq a$ we have $x<b$.
    Similarly, for all $y\in Y^{{s}}$ with $s>i$ and $y\neq b$, we get $y>a$.
    Thus, all blocks satisfy the relations of $P_{\tau'}$ with
    \begin{equation*}
    \tau'=(\tau,\dots, \tau_{i-1},\tau_{i}-1,1,\tau_{i+1}-1,\dots,\tau_{\ell})
    \end{equation*}
    and thus the quotient poset is just $P_{\tau'}$.
    If $i=\ell$, the same arguments yield that the quotient poset is isomorphic to $P_{\tau'}$ with
    \begin{equation*}
    \tau'=(\tau,\dots, \tau_{\ell}-1)
    \end{equation*}
    Now, if a zero entry should arise in the tuples above, then we just skip it.
\end{proof}

Now we will inspect the chain inequalities, where $C\neq\{\hat{0}\}$.

\begin{proposition}\label{Proposition-Zero-Variable-Face}
   %
   Let $F\subseteq\OP_{C,O}(\tau)$ be a face given by $x_b=0$ for some $b\in C\setminus\{\hat{0}\}$ from a $k$-decomposition $P_\tau=C \sqcup O$.
   For $b\in Y^i$ we have $F\cong\OP_{\Tilde{C},O}(\Tilde{\tau})$
   where
   \begin{equation*}
     \Tilde{\tau}=(\tau_1,\dots,\tau_{i-1},\tau_{i}-1,\tau_{i+1},\dots,\tau_{\ell})
   \end{equation*}
   and $\Tilde{C}=C\setminus\{b\}$.
   Again, zero entries in $\Tilde{\tau}$ are omitted.
\end{proposition}
\begin{proof}
  The chain-inequalities using $b$ are of the form $\sum_{j=1}^{k}x_{p_{j}}\leq x_{a}$  with $b=p_{i}$.
  For $x_b=0$ such an inequality simplifies to
  \begin{equation*}
    \sum\limits_{\substack{j=1\\j\neq i}}^{k}x_{p_{i}}\leq x_{a} \enspace .
  \end{equation*}
  If there exists $c\in Y^{{i}}\setminus\{b\}$, then these inequalities are redundant in view of
  \begin{equation*}
    x_{c}+\sum\limits_{\substack{j=1\\j\neq i}}^{k}x_{p_{i}}\leq x_{a} \enspace .
  \end{equation*}
  Either way, the defining inequalities of the face are exactly the defining inequalities of $\OP_{\Tilde{C}O}(\Tilde{\tau})$.
  Thus, the linear projection
  \[
    F\rightarrow\OP_{\Tilde{C},O}(\Tilde{\tau}) \,,\ (x_{p})_{p\in P_\tau}\mapsto (x_{p})_{p\in P_{\Tilde{\tau}}}
  \]
  is an affine isomorphism.
\end{proof}
\begin{proposition}\label{Proposition-Chains-Codimension}
    Let $F\subseteq\OP_{C,O}(\tau)$ be the face for a $k$-decomposition given by the equalities from the maximal chains
\begin{equation*}
    \mathbf{p}=\{\hat{0},p_1,\dots,p_k,a\},\hspace{1em}\mathbf{p'}=\{\hat{0},p'_1,\dots,p'_k,a'\}
\end{equation*}
where $a$ and $a'$ belong to the minimal antichain of $O$. 
We define $p_{k+1}:= a$ and $p'_{k+1}:= a'$. 
Let $I\subseteq\{1,\dots, k+1\}$ be the set of indices, where we have $p_i\neq p'_i$ for $i\in I$.
By projecting the coordinates of $\mathbf{p}'\setminus\mathbf{p}$ away, we get the affine isomorphism
\[
  \nu:F\rightarrow F' \,,\ (x_{p})_{p\in P_\tau}\mapsto (x_{p'})_{p'\in P_{\tau'}} \enspace ,
\]
where $\tau'=(\tau'_1,\dots,\tau'_\ell)$ with
\begin{equation*}
    \tau'_j=\begin{cases}
        \tau_j-1&\text{ if $j\in I$}\\
        \tau_j&\text{ otherwise}
    \end{cases}
\end{equation*}
and the face $F'$ is given by the equation of $\mathbf{p}$. Each coordinate that gets projected away yields us exactly one additional codimension.
\end{proposition}
\begin{proof}
Fix an $i\in I$ with $i<k+1$. Then we can form the two maximal chains
\begin{equation*}
    \mathbf{q}=(\mathbf{p}\setminus\{p_i\})\cup\{p'_i\},\hspace{1em}\mathbf{q}'=(\mathbf{p}'\setminus\{p'_i\})\cup\{p_i\}
\end{equation*}
Then, we get the inequalities
\[
  x_a+x_{p'_{i}}-x_{p_{i}}=\sum\limits_{q\in\mathbf{q}\cap C}x_q\leq x_a \quad \text{and} \quad  x_{a'}+x_{p_{i}}-x_{p'_{i}}=\sum\limits_{q\in\mathbf{q}'\cap C}x_q\leq x_{a'} \enspace .
\]
In particular, we get $x_{p'_{i}}\leq x_{p_{i}}$ as well as $x_{p_{i}}\leq x_{p'_{i}}$. 
That is, we have $x_{p_{i}}=x_{p'_{i}}$. 
If we also have $a\neq a'$, then using the chains
\begin{equation*}
    \mathbf{r}=(\mathbf{p}\setminus\{a\})\cup\{a'\},\hspace{1em}\mathbf{r}'=(\mathbf{p}'\setminus\{a'\})\cup\{a\}
\end{equation*}
we obtain
\[
  x_a=\sum\limits_{q\in\mathbf{p}\cap C}x_q=\sum\limits_{q\in\mathbf{r}\cap C}x_q\leq x_{a'} \quad \text{and} \quad  x_{a'}=\sum\limits_{q\in\mathbf{p}'\cap C}x_q=\sum\limits_{q\in\mathbf{r}'\cap C}x_q\leq x_{a} \enspace .
\]
We infer $x_a=x_{a'}$.
So the face $F$ satisfies the equation given by $\mathbf{p}$ and the equations
\begin{equation}\label{eq-diff-chains}
    x_{p_{i}}=x_{p'_{i}}\text{ for all $i\in I$} \enspace .
\end{equation}
Further, the subset of $\OP_{C,O}(\tau)$ given by the equation by $\mathbf{p}$ and the equalities \eqref{eq-diff-chains} satisfies the equation given by $\mathbf{p}'$, i.e. the equations of \eqref{eq-diff-chains} and $\mathbf{p}$ satisfy $F$. 
We can thus project the coordinates from $\mathbf{p}'\setminus\mathbf{p}$ away and get the claimed affine isomorphism.
\end{proof}

We can now establish a normal form for the faces of the chain-order polytope.

\begin{theorem}\label{Theorem-ChainOrder-Normal}
  Consider a $k$-decomposition $P_\tau=C\sqcup O$.
  Any face $F\subseteq\OP_{C,O}(\tau)$ is uniquely given by:
\begin{enumerate}
\item{a face partition of $O$, which we call $\pi$};
\item{sets $F_{0,i}\subseteq Y^i$ for $1\leq i\leq k$ (indicating those variables which are zero for all points in $F$);}
  \item\label{thm:chainorder-chain-cond}{
      sets $F_{\mathrm{eq},i}$ for $1\leq i\leq k+1$ (indicating the tight chain inequalities) satisfying
      \begin{enumerate}
      \item{
          $F_{\mathrm{eq},i}\subseteq Y^{i}\setminus F_{0,i}$, for all $1\leq i\leq k$,
        }
      \item{
          if $Y^{{i}}\setminus F_{0,i}\neq\emptyset$, then $\emptyset\neq F_{\mathrm{eq},i}$, for all $1\leq i\leq k$,
        }
      \item{
          $F_{\mathrm{eq},k+1}\subseteq Y^{k+1}$,
        }
      \item{
          $F_{\mathrm{eq},k+1}$ only contains elements of $Y^{k+1}$ that belong to singleton blocks of the face partition $\pi$.
        }
      \end{enumerate}
    }
  \end{enumerate}
\end{theorem}
\begin{proof}
The face partition results from Proposition \ref{Proposition-Face-Partition}.

We now fix the set of variables equal to $0$, which we call $F_{0,i}$ for all $1\leq i\leq n$.
These can be projected away according to Proposition \ref{Proposition-Zero-Variable-Face}.
Recall that the variables correspond to the elements of the poset $P$.
If we now want to add chains, then each chain has to contain an element $y\in Y^{{i}}\setminus F_{0,i}$ for all $1\leq i\leq k$ whenever $Y^{{i}}\setminus F_{0,i}\neq\emptyset$.
Otherwise, there is no choice we can make for $Y^{{i}}$.

Then, all chains have to end in the minimal elements of $\Ptau/\pi$ (rigorously speaking, we must put the remaining elements in singleton blocks to form the quotient). 
These minimal elements are the singletons blocks of $Y^{{k+1}}$ if they exist. 
If all elements of $Y^{{k+1}}$ are in blocks containing elements greater than them, and thus in the same block, there is also no choice for us to make.
\end{proof}

With the notation from Theorem \ref{Theorem-ChainOrder-Normal}, and employing Propositions \ref{prop-normal-form-order}, \ref{Proposition-Zero-Variable-Face} and \ref{Proposition-Chains-Codimension}, the codimension of the face $F$ can be computed as follows.
If there are no tight chain inequalities from \ref{thm:chainorder-chain-cond}, then
\begin{equation}\label{eq:codim:no-tight-chain}
  \codim F = |O|-|\pi|+\sum\limits_{i=1}^{n}|F_{0,i}| \enspace .
\end{equation}
Otherwise, we have
\begin{equation}\label{eq:codim:other}
  \codim F = 1+\sum\limits_{i\in I}\left(|F_{\mathrm{eq},i}|-1\right)+|O|-|\pi|+\sum\limits_{i=1}^{n}|F_{0,i}| \enspace,
\end{equation}
where $I\subseteq\{1,\dots,k+1\}$ is the set of indices for which $F_{\mathrm{eq},i}\neq\emptyset$.
From the special case $P_{\tau}=P_{\tau}\sqcup\emptyset$, we get a face description of the chain polytope.
\begin{corollary}\label{cor-chain-normal}
Let $\tau=(\tau_1,\dots,\tau_\ell)\in\Z_{> 0}^\ell$ and $F\subseteq\CP(\tau)$ be a face, then it is uniquely given by:
\begin{enumerate}
    \item{
        Sets $F_{0,i}\subseteq Y^i$ for $1\leq i\leq \ell$ (indicating the order inequalities tight at $F$);}
      \item\label{thm:chain-chain-cond}{
          sets $F_{\mathrm{eq},i}$ for $1\leq i\leq k+1$ (indicating the tight chain inequalities) satisfying
          \begin{enumerate}
            \item{
                $F_{\mathrm{eq},i}\subseteq Y^{i}\setminus F_{0,i}$, for all $1\leq i\leq \ell$,
            }
            \item{
                if $Y^{{i}}\setminus F_{0,i}\neq\emptyset$, then $\emptyset\neq F_{\mathrm{eq},i}$, for all $1\leq i\leq \ell$.
            }
        \end{enumerate}
    }
  \end{enumerate}
\end{corollary}

The codimension formulae \eqref{eq:codim:no-tight-chain},\eqref{eq:codim:other} yield the following.
Either no chain inequalities are tight at $F$.
Then
\begin{equation*}
  \codim F = \sum\limits_{i=1}^{n}|F_{0,i}| \enspace .
\end{equation*}
Or at least one chain inequality is tight, in which case
\begin{equation*}
  \codim F = 1+\sum\limits_{i\in I}\left(|F_{\mathrm{eq},i}|-1\right)+\sum\limits_{i=1}^{n}|F_{0,i}| \enspace ,
\end{equation*}
where $I\subseteq\{1,\dots,k+1\}$ is the set of indices for which $F_{\mathrm{eq},i}\neq\emptyset$.

\begin{figure}
    \centering
    \adjustbox{scale=0.7}{
    \begin{tikzcd}[ampersand replacement=\&,execute at end picture={
        \draw [red]
            ($(5-3)+(-3, +1.8em)$) -- ($(5-3)+(3, +1.8em)$);
        \fill[red,opacity=0.3] \convexpath{8-3,7-2,6-2,5-3,4-1,5-3,6-2,7-2,8-3}{8pt};
        \fill[red,opacity=0.3] \convexpath{8-3,7-2,6-2,5-3,4-3,5-3,6-2,7-2,8-3}{8pt};
        \fill[blue,opacity=0.3]  (6-4) circle (.265);
        \fill[green,opacity=0.3] \convexpath{4-2,3-2,3-4}{8pt};
        }]
    	\&\&\&\& |[alias=1-3]|{\hat{1}} \\
    	\& {\tau_6=3} \& |[alias=2-1]|\bullet \&\& |[alias=2-3]|\bullet \& {} \& |[alias=2-5]|\bullet \\
    	{O} \& {\tau_5=2} \&\& |[alias=3-2]|\bullet \&\& |[alias=3-4]|\bullet \\
    	\& {\tau_4=4} \& |[alias=4-1]|\bullet \& |[alias=4-2]|\bullet \&\& |[alias=4-3]|\bullet \& |[alias=4-5]|\bullet \\
    	\& {\tau_3=1} \&\&\& |[alias=5-3]|\bullet \\
    	{C} \& {\tau_2=2} \&\& |[alias=6-2]|\bullet \&\& |[alias=6-4]|\bullet \\
    	\& {\tau_1=5} \& |[alias=7-1]|\bullet \& |[alias=7-2]|\bullet \& |[alias=7-3]|\bullet \& |[alias=7-4]|\bullet \& |[alias=7-5]|\bullet \\
    	\&\&\&\& |[alias=8-3]|{\hat{0}}
    	\arrow[no head, from=2-3, to=1-5]
    	\arrow[no head, from=2-5, to=1-5]
    	\arrow[no head, from=2-7, to=1-5]
    	\arrow[no head, from=7-3, to=6-4]
    	\arrow[no head, from=7-4, to=6-4]
    	\arrow[no head, from=7-5, to=6-4]
    	\arrow[no head, from=7-6, to=6-4]
    	\arrow[no head, from=7-7, to=6-4]
    	\arrow[no head, from=7-3, to=6-6]
    	\arrow[no head, from=7-4, to=6-6]
    	\arrow[no head, from=7-5, to=6-6]
    	\arrow[no head, from=7-6, to=6-6]
    	\arrow[no head, from=7-7, to=6-6]
    	\arrow[no head, from=8-5, to=7-3]
    	\arrow[no head, from=8-5, to=7-4]
    	\arrow[no head, from=8-5, to=7-5]
    	\arrow[no head, from=8-5, to=7-6]
    	\arrow[no head, from=8-5, to=7-7]
    	\arrow[no head, from=6-4, to=5-5]
    	\arrow[no head, from=6-6, to=5-5]
    	\arrow[no head, from=5-5, to=4-3]
    	\arrow[no head, from=5-5, to=4-4]
    	\arrow[no head, from=5-5, to=4-6]
    	\arrow[no head, from=5-5, to=4-7]
    	\arrow[no head, from=4-3, to=3-4]
    	\arrow[no head, from=4-4, to=3-4]
    	\arrow[no head, from=4-6, to=3-4]
    	\arrow[no head, from=4-7, to=3-4]
    	\arrow[no head, from=4-3, to=3-6]
    	\arrow[no head, from=4-4, to=3-6]
    	\arrow[no head, from=4-6, to=3-6]
    	\arrow[no head, from=4-7, to=3-6]
    	\arrow[no head, from=3-4, to=2-3]
    	\arrow[no head, from=3-4, to=2-5]
    	\arrow[no head, from=3-4, to=2-7]
    	\arrow[no head, from=3-6, to=2-3]
    	\arrow[no head, from=3-6, to=2-5]
    	\arrow[no head, from=3-6, to=2-7]
    \end{tikzcd}
    }
    \caption{
        Normal form of a face of $\OP_{C,O}((5,2,1,4,2,3))$ with $k=3$ of codimension $5$.
        The elements in red belong to $F_{\mathrm{eq},i}$ for respective $i$, the element in blue belongs to $F_{0,2}$ and the green block is a block of cardinality $3$ in a face partition of $O\cup\{\hat{1}\}$. Elements in $O$ that are not marked are in singletons.
    }
    \label{ex:running}
\end{figure}

\section{Proof of the main theorem}\label{sec-hibili}
In this section we use Theorem \ref{Theorem-ChainOrder-Normal} to obtain our main result:

\begin{theorem}\label{Hibi-Li-Claim-Crucial}
  Let $\tau\in \Z_{> 0}^\ell$ and let $0\leq k\leq n-1$.
  We consider the $k$-decomposition $P_\tau=C\sqcup O$ and the $(k+1)$-decomposition $P_\tau=C'\sqcup O'$.

  Then, there exists an injective map from the set of faces of codimension $i$ of $\OP_{C,O}(\tau)$ to the set of faces of codimension $i$ of $\OP_{C',O'}(\tau)$ for all $2\leq i\leq |\tau|$.
\end{theorem}
The crucial insight is that the normal forms of the $k$- and $(k+1)$-decompositions differ only on the set $Y^{k+1}\cup Y^{k+2}$.
Thus, we construct the proof based on the case distinction of the face partition of the $k$-decomposition.
Proving the above result establishes the Hibi--Li inequality.
\begin{corollary}
  For $\tau\in \Z_{> 0}^\ell$ we have
  \begin{equation*}
    f_i(\OP(\tau))\leq f_i(\CP(\tau)) \quad \text{for all } 1\leq i\leq|\tau| \enspace .
  \end{equation*}
\end{corollary}
\begin{proof}
  The claim follows from Theorem~\ref{Hibi-Li-Claim-Crucial} since $\OP(\tau)$ is a chain-order polytope for the $0$-partition, and $\CP(\tau)$ to that of the $n$-partition.
\end{proof}

\subsection{Construction: Preliminaries}
We now define a map, $\psi$, which sends faces of $\OP_{C,O}(\tau)$ to faces of $\OP_{C',O'}(\tau)$.
We first introduce some technical details that will be used in the construction, then we will define $\psi$ by distinguishing several cases.
Finally, in Section~\ref{sec:injective} we will see that $\psi$ is injective.
Along the way, we exploit that the normal forms differ only locally.
Let us fix a face $F$ of $\OP_{C,O}(\tau)$ with $\codim F = t$. We construct a corresponding face $F'=\psi(F)$ of $\OP_{C',O'}(\tau)$ of the same codimension.
\subsubsection{Face partition:}
The idea is to \enquote{cut off} the elements $Y^{k+1}$ from all blocks $\pi$ of the face partition of $O\cup\{\hat{1}\}$. 
We look at the partition
\begin{equation*}
  \pi'':=\{K\cap O'\mid K\in \pi, K\not\subseteq Y^{{k+1}}\cup Y^{{k+2}}\}
\end{equation*}
and define the extended partition $\pi'$ by putting the remaining elements of $O'$ (those which are not contained in any block of $\pi''$) into singletons.
The following result shows that $\pi'$ is a face partition.
\begin{lemma}
  The partition $\pi''$ is connected, and thus also $\pi'$.
\end{lemma}
\begin{proof}
  If $K\cap O'=K$ for $K\in \pi$, then this block is, of course, connected. 
  If $K\cap O'\neq K$, then $K$ contains elements from $Y^{{k+1}}$. 
  However, by construction, we require that $K$ also contains elements greater than $Y^{{k+2}}$ and thus $K$ contains $Y^{{k+2}}$ ($K$ is convex~{\cite[Remark 3.8]{Peg18}}), which allows it to be connected by forming the fence $a>y<b$ for $a,b\in K\cap O'$ and some $y\in Y^{k}$. 
  Thus $K\cap O'$ is connected.
  Hence, we have that $\pi'$ is connected, and of course, it is again compatible and a face partition.
\end{proof}

One might ask why we cannot just intersect with $O'$ and leave out the condition on $K$. 
The reason is that a block $K\in\pi$ of height $1$ would result in a disconnected block if we intersect it with $O'$.
Hence, we must leave out the minimal blocks of height $1$.
Therefore, we need to put the restriction on the block $K$.

\subsubsection{Zero elements:}
Here, we just copy the zero elements. We set
\begin{equation*}
  F'_{0,i}=F_{0,i}
\end{equation*}
for all $1\leq i\leq k$. 
The variables for $i=k+1$ will be considered in the individual cases.

\subsubsection{Chain elements:}
Similarly, as before, we simply copy the chain elements.
If $F$ already uses chains, then we set
\begin{equation*}
  F'_{\mathrm{eq},i}=F_{\mathrm{eq},i}
\end{equation*}
for all $1\leq i\leq k+1$. 
The chain variables for $i=k+2$ will be considered in the individual cases.

\subsection{Construction: Generic Case}
Now, we distinguish between two main cases.
First, we consider the situation where the face $F$ already uses chains; this is the generic case.
We must determine which elements of $Y^{k+1}$ are in singleton blocks, if there are any.
Let $r\in\{1,\ldots,\tau_{k+2}\}$ be minimal with the property that $y_{r}^{{k+2}}$ is a singleton, if it exists.
To make choices for the chain elements in $Y^{k+2}$, we consider two subcases.

\subsubsection{All elements of $Y^{{k+1}}$ are in singletons in $\pi$}\label{Case-Gen-Single}
Here, we do not cut off any block. Thus, we simply continue the chains and define
\begin{equation*}
    F'_{0,k+1}=\emptyset,\quad \text{and} \quad F'_{\mathrm{eq},k+2}=\{y_{r}^{{k+2}}\} \enspace .
\end{equation*}
Naturally, in case of $k=\ell-1$, the set $F'_{\mathrm{eq},k+2}$ cannot be defined and we only set the zero variables.
The codimension also match as we did not alter the non-singleton blocks in $F$ and did not remove nor add chain or zero variables.

\subsubsection{Not all elements of $Y^{{k+1}}$ are in singletons blocks in $\pi$}\label{Case-Gen-Non-Single}
In this case, let $\{a_1,\dots, a_l\}$ be the set of elements of $Y^{{k+1}}$ contained in blocks with elements greater than themselves. 
These elements must be in the same block, which we call $K\in\pi$; note that $K$ is convex by \cite[Remark 3.8]{Peg18}.
We define
\begin{equation}\label{Eq-Min-Block-Decomp}
    A:=K\cap Y^{k+1} \quad \text{and} \quad B:=K\cap Y^{k+2} \enspace .
\end{equation}
With these two sets, we set
\begin{equation*}
  F'_{0,k+1}=A \enspace.
\end{equation*}
If $K$ is of height $1$, we further set
\begin{equation*}
  F'_{\mathrm{eq},k+2}=B \enspace .
\end{equation*}
The codimensions remain consistent, as we appropriately set the required number of variables to 0 and assign the necessary chain variables when applicable.
\begin{figure}
    \centering
    \adjustbox{scale=0.5}{
        \begin{tikzcd}[ampersand replacement=\&,execute at end picture={
        \draw [red]
            ($(1-5-3)+(-3, +1.8em)$) -- ($(1-5-3)+(3, +1.8em)$);
        \draw [red]
            ($(2-4-3)+(-3, +1.8em)$) -- ($(2-4-3)+(3, +1.8em)$);
        \fill[green,opacity=0.3] \convexpath{1-4-1,1-3-4,1-4-2}{8pt};
        \fill[orange,opacity=0.3] \convexpath{1-1-3,1-2-5,1-2-3}{8pt};
        \fill[orange,opacity=0.3] \convexpath{2-1-3,2-2-5,2-2-3}{8pt};
        \fill[red,opacity=0.3] \convexpath{1-8-3,1-7-3,1-6-2,1-5-3,1-4-4,1-5-3,1-6-2,1-7-3}{8pt};
        \fill[red,opacity=0.3] \convexpath{1-8-3,1-7-3,1-6-4,1-5-3,1-4-5,1-5-3,1-6-4,1-7-3}{8pt};
        \fill[red,opacity=0.3] \convexpath{2-8-3,2-7-3,2-6-2,2-5-3,2-4-4,2-3-4,2-4-4,2-5-3,2-6-2,2-7-3}{8pt};
        \fill[red,opacity=0.3] \convexpath{2-8-3,2-7-3,2-6-4,2-5-3,2-4-5,2-3-4,2-4-5,2-5-3,2-6-4,2-7-3}{8pt};
        \fill[blue,opacity=0.3]  (1-7-4) circle (.265);
        \fill[blue,opacity=0.3]  (1-7-5) circle (.265);
        \fill[blue,opacity=0.3]  (2-7-4) circle (.265);
        \fill[blue,opacity=0.3]  (2-7-5) circle (.265);
        \fill[blue,opacity=0.3]  (2-4-1) circle (.265);
        \fill[blue,opacity=0.3]  (2-4-2) circle (.265);
        }]
        \&\&\&\& |[alias=1-1-3]|{\hat{1}} \&\&\&\&\&\&\&\&\&\&\& |[alias=2-1-3]|{\hat{1}} \\
        \& {\tau_6=3} \& |[alias=1-2-1]|\bullet \&\& |[alias=1-2-3]|\bullet \& {} \& |[alias=1-2-5]|\bullet \&\&\&\&\&\& {\tau_6=1} \& |[alias=2-2-1]|\bullet \&\& |[alias=2-2-3]|\bullet \&\& |[alias=2-2-5]|\bullet \\
        {O} \& {\tau_5=2} \&\& |[alias=1-3-2]|\bullet \&\& |[alias=1-3-4]|\bullet \&\&\&\&\&\& {O'} \& {\tau_5=2} \&\& |[alias=2-3-2]|\bullet \&\& |[alias=2-3-4]|\bullet \\
        \& {\tau_4=4} \& |[alias=1-4-1]|\bullet \& |[alias=1-4-2]|\bullet \&\& |[alias=1-4-4]|\bullet \& |[alias=1-4-5]|\bullet \&\&\&\&\&\& {\tau_4=4} \& |[alias=2-4-1]|\bullet \& |[alias=2-4-2]|\bullet \&|[alias=2-4-3]|{}\& |[alias=2-4-4]|\bullet \& |[alias=2-4-5]|\bullet \\
        \& {\tau_3=1} \&\&\& |[alias=1-5-3]|\bullet \&\&\& {} \&\&\& {} \&\& {\tau_3=1} \&\&\& |[alias=2-5-3]|\bullet \\
        {C} \& {\tau_2=2} \&\& |[alias=1-6-2]|\bullet \&\& |[alias=1-6-4]|\bullet \&\&\&\&\&\& {C'} \& {\tau_2=2} \&\& |[alias=2-6-2]|\bullet \&\& |[alias=2-6-4]|\bullet \\
        \& {\tau_1=5} \& |[alias=1-7-1]|\bullet \& |[alias=1-7-2]|\bullet \& |[alias=1-7-3]|\bullet \& |[alias=1-7-4]|\bullet \& |[alias=1-7-5]|\bullet \&\&\&\& {} \&\& {\tau_1=4} \& |[alias=2-7-1]|\bullet \& |[alias=2-7-2]|\bullet \& |[alias=2-7-3]|\bullet \& |[alias=2-7-4]|\bullet \& |[alias=2-7-5]|\bullet \\
        \&\&\&\& |[alias=1-8-3]|{\hat{0}} \&\&\&\&\&\&\&\&\&\&\& |[alias=2-8-3]|{\hat{0}}
        \arrow[no head, from=2-3, to=1-5]
        \arrow[no head, from=2-5, to=1-5]
        \arrow[no head, from=2-7, to=1-5]
        \arrow[no head, from=7-3, to=6-4]
        \arrow[no head, from=7-4, to=6-4]
        \arrow[no head, from=7-5, to=6-4]
        \arrow[no head, from=7-6, to=6-4]
        \arrow[no head, from=7-7, to=6-4]
        \arrow[no head, from=7-3, to=6-6]
        \arrow[no head, from=7-4, to=6-6]
        \arrow[no head, from=7-5, to=6-6]
        \arrow[no head, from=7-6, to=6-6]
        \arrow[no head, from=7-7, to=6-6]
        \arrow[no head, from=8-5, to=7-3]
        \arrow[no head, from=8-5, to=7-4]
        \arrow[no head, from=8-5, to=7-5]
        \arrow[no head, from=8-5, to=7-6]
        \arrow[no head, from=8-5, to=7-7]
        \arrow[no head, from=6-4, to=5-5]
        \arrow[no head, from=6-6, to=5-5]
        \arrow[no head, from=5-5, to=4-3]
        \arrow[no head, from=5-5, to=4-4]
        \arrow[no head, from=5-5, to=4-6]
        \arrow[no head, from=5-5, to=4-7]
        \arrow[no head, from=4-3, to=3-4]
        \arrow[no head, from=4-4, to=3-4]
        \arrow[no head, from=4-6, to=3-4]
        \arrow[no head, from=4-7, to=3-4]
        \arrow[no head, from=4-3, to=3-6]
        \arrow[no head, from=4-4, to=3-6]
        \arrow[no head, from=4-6, to=3-6]
        \arrow[no head, from=4-7, to=3-6]
        \arrow[no head, from=3-4, to=2-3]
        \arrow[no head, from=3-4, to=2-5]
        \arrow[no head, from=3-4, to=2-7]
        \arrow[no head, from=3-6, to=2-3]
        \arrow[no head, from=3-6, to=2-5]
        \arrow[no head, from=3-6, to=2-7]
        \arrow[from=5-8, to=5-11]
        \arrow[no head, from=5-16, to=4-14]
        \arrow[no head, from=5-16, to=4-18]
        \arrow[no head, from=8-16, to=7-14]
        \arrow[no head, from=8-16, to=7-18]
        \arrow[no head, from=6-15, to=7-14]
        \arrow[no head, from=6-15, to=7-18]
        \arrow[no head, from=6-17, to=7-14]
        \arrow[no head, from=6-17, to=7-18]
        \arrow[no head, from=5-16, to=6-15]
        \arrow[no head, from=5-16, to=6-17]
        \arrow[no head, from=4-14, to=3-15]
        \arrow[no head, from=4-15, to=3-15]
        \arrow[no head, from=4-17, to=3-15]
        \arrow[no head, from=4-18, to=3-15]
        \arrow[no head, from=4-14, to=3-17]
        \arrow[no head, from=4-15, to=3-17]
        \arrow[no head, from=4-17, to=3-17]
        \arrow[no head, from=4-18, to=3-17]
        \arrow[no head, from=5-16, to=4-15]
        \arrow[no head, from=5-16, to=4-17]
        \arrow[no head, from=3-15, to=2-14]
        \arrow[no head, from=2-14, to=1-16]
        \arrow[no head, from=2-16, to=1-16]
        \arrow[no head, from=2-18, to=1-16]
        \arrow[no head, from=3-17, to=2-18]
        \arrow[no head, from=3-15, to=2-16]
        \arrow[no head, from=3-17, to=2-16]
        \arrow[no head, from=3-17, to=2-14]
        \arrow[no head, from=3-15, to=2-18]
        \arrow[no head, from=8-16, to=7-17]
        \arrow[no head, from=8-16, to=7-15]
        \arrow[no head, from=7-15, to=6-15]
        \arrow[no head, from=7-15, to=6-17]
        \arrow[no head, from=7-17, to=6-17]
        \arrow[no head, from=7-17, to=6-15]
        \arrow[no head, from=7-16, to=6-15]
        \arrow[no head, from=7-16, to=6-17]
        \end{tikzcd}
    }
    \caption{
        Example of construction in Case \ref{Case-Gen-Non-Single} when $F$ already uses a chain and $K$ is of height $1$. The orange blocks are blocks of the face partitions of $O\cup\{\hat{1}\}$ and $O'\cup\{\hat{1}\}$, respectively. For the remaining color-coding, see figure \ref{ex:running}.
    }
\end{figure}
\subsection{Construction: Degenerate Case}
Now, we consider the degenerate cases where the face $F$ does not use any chains.
It turns out that the previous construction from the generic case would fail, as it would increase the codimension of the face constructed by one.
Hence, we will look closer into the chain variables.
Once again, we arrive at two subcases.
\subsubsection{All elements of $Y^{{k+1}}$ are in singletons in $\pi$}\label{Case-Deg-Single}
Here, we only set
\begin{equation*}
    F'_{0,k+1}=\emptyset \enspace .
\end{equation*}
As in the generic case above, we can verify that the codimensions match.
\subsubsection{Not all elements of $Y^{{k+1}}$ are in singletons in $\pi$}
This is the most delicate case.
Again, we let $K\in\pi$ be the block containing the set $A\subseteq Y^{k+1}$.
To proceed, we further distinguish different configurations of $K$.

\paragraph{Assume that the height of $K$ is at least two}\label{Case-Deg-Height-2}
Here, we can set 
\begin{equation*}
    F'_{0,k+1}=A \enspace .
\end{equation*}
Once again the codimensions match.
\paragraph{Assume that the height of $K$ equals one}\label{Case-Deg-Height-1}
Here, we have
\begin{equation*}
    K=A\cup B
\end{equation*}
with $A$ and $B$ from \eqref{Eq-Min-Block-Decomp}.
For a refined argument, we now additionally consider the properties of set $B$.
\subparagraph{Additionally assume that $B\neq\{y_{r}^{k+2}\}$}\label{Case-Deg-Height-1-Max-Diff}{%
  Here, we set
  \begin{equation*}
    F'_{0,k+1}=\emptyset \enspace . 
  \end{equation*}
  For each $1\leq i\leq k$ choose $y^{i}_{t}\in Y^{{i}}\setminus F_{0,i}$ with $1\leq t\leq\tau_i$ minimal, if the set is non-empty.
  These elements form a chain of length $k$, which we assign to $F'_{\mathrm{eq},i}$.
  This chain allows to identify $K$.
  With
  \begin{equation*}
    F'_{\mathrm{eq},k+1}=A, \quad \text{and} \quad F'_{\mathrm{eq},k+2}=B
  \end{equation*}
  we obtain the equality
  \begin{equation*}
    1+|A|-1+|B|-1=|A|+|B|-1=|K|-1 \enspace.
  \end{equation*}
  This shows that $F'$ has the same codimension as $F$.%
}
\subparagraph{Additionally assume that $B=\{y_{r}^{{k+2}}\}$}\label{Case-Deg-Height-1-Max-Standard}{%
  Here we need to avoid a conflict with Case~\ref{Case-Gen-Single}.
  Thus, we set
  \begin{equation*}
    F'_{0,k+1}=A \enspace .
  \end{equation*}
  Since $|B|=1$, we obtain the correct codimension for $F'$.
}
\begin{figure}
    \centering
    \adjustbox{scale=0.5}{
        \begin{tikzcd}[ampersand replacement=\&,execute at end picture={
            \draw [red]
                ($(1-5-3)+(-3, +1.8em)$) -- ($(1-5-3)+(3, +1.8em)$);
            \draw [red]
                ($(2-4-3)+(-3, +1.8em)$) -- ($(2-4-3)+(3, +1.8em)$);
            \fill[green,opacity=0.3] \convexpath{1-4-1,1-3-4,1-4-2}{8pt};
            \fill[orange,opacity=0.3] \convexpath{1-1-3,1-2-5,1-2-3}{8pt};
            \fill[orange,opacity=0.3] \convexpath{2-1-3,2-2-5,2-2-3}{8pt};
            \fill[red,opacity=0.3] \convexpath{2-8-3,2-7-1,2-6-2,2-5-3,2-4-2,2-3-4,2-4-2,2-5-3,2-6-2,2-7-1}{8pt};
            \fill[red,opacity=0.3] \convexpath{2-8-3,2-7-1,2-6-2,2-5-3,2-4-1,2-3-4,2-4-1,2-5-3,2-6-2,2-7-1}{8pt};
            \fill[blue,opacity=0.3]  (1-7-4) circle (.265);
            \fill[blue,opacity=0.3]  (1-7-5) circle (.265);
            \fill[blue,opacity=0.3]  (2-7-4) circle (.265);
            \fill[blue,opacity=0.3]  (2-7-5) circle (.265);
            }]
        	\&\&\&\& |[alias=1-1-3]|{\hat{1}} \&\&\&\&\&\&\&\&\&\&\& |[alias=2-1-3]|{\hat{1}} \\
        	\& {\tau_6=3} \& |[alias=1-2-1]|\bullet \&\& |[alias=1-2-3]|\bullet \& {} \& |[alias=1-2-5]|\bullet \&\&\&\&\&\& {\tau_6=1} \& |[alias=2-2-1]|\bullet \&\& |[alias=2-2-3]|\bullet \&\& |[alias=2-2-5]|\bullet \\
        	{O} \& {\tau_5=2} \&\& |[alias=1-3-2]|\bullet \&\& |[alias=1-3-4]|\bullet \&\&\&\&\&\& {O'} \& {\tau_5=2} \&\& |[alias=2-3-2]|\bullet \&\& |[alias=2-3-4]|\bullet \\
        	\& {\tau_4=4} \& |[alias=1-4-1]|\bullet \& |[alias=1-4-2]|\bullet \&\& |[alias=1-4-4]|\bullet \& |[alias=1-4-5]|\bullet \&\&\&\&\&\& {\tau_4=4} \& |[alias=2-4-1]|\bullet \& |[alias=2-4-2]|\bullet \&|[alias=2-4-3]|{}\& |[alias=2-4-4]|\bullet \& |[alias=2-4-5]|\bullet \\
        	\& {\tau_3=1} \&\&\& |[alias=1-5-3]|\bullet \&\&\& {} \&\&\& {} \&\& {\tau_3=1} \&\&\& |[alias=2-5-3]|\bullet \\
        	{C} \& {\tau_2=2} \&\& |[alias=1-6-2]|\bullet \&\& |[alias=1-6-4]|\bullet \&\&\&\&\&\& {C'} \& {\tau_2=2} \&\& |[alias=2-6-2]|\bullet \&\& |[alias=2-6-4]|\bullet \\
        	\& {\tau_1=5} \& |[alias=1-7-1]|\bullet \& |[alias=1-7-2]|\bullet \& |[alias=1-7-3]|\bullet \& |[alias=1-7-4]|\bullet \& |[alias=1-7-5]|\bullet \&\&\&\& {} \&\& {\tau_1=4} \& |[alias=2-7-1]|\bullet \& |[alias=2-7-2]|\bullet \& |[alias=2-7-3]|\bullet \& |[alias=2-7-4]|\bullet \& |[alias=2-7-5]|\bullet \\
        	\&\&\&\& |[alias=1-8-3]|{\hat{0}} \&\&\&\&\&\&\&\&\&\&\& |[alias=2-8-3]|{\hat{0}}
        	\arrow[no head, from=2-3, to=1-5]
        	\arrow[no head, from=2-5, to=1-5]
        	\arrow[no head, from=2-7, to=1-5]
        	\arrow[no head, from=7-3, to=6-4]
        	\arrow[no head, from=7-4, to=6-4]
        	\arrow[no head, from=7-5, to=6-4]
        	\arrow[no head, from=7-6, to=6-4]
        	\arrow[no head, from=7-7, to=6-4]
        	\arrow[no head, from=7-3, to=6-6]
        	\arrow[no head, from=7-4, to=6-6]
        	\arrow[no head, from=7-5, to=6-6]
        	\arrow[no head, from=7-6, to=6-6]
        	\arrow[no head, from=7-7, to=6-6]
        	\arrow[no head, from=8-5, to=7-3]
        	\arrow[no head, from=8-5, to=7-4]
        	\arrow[no head, from=8-5, to=7-5]
        	\arrow[no head, from=8-5, to=7-6]
        	\arrow[no head, from=8-5, to=7-7]
        	\arrow[no head, from=6-4, to=5-5]
        	\arrow[no head, from=6-6, to=5-5]
        	\arrow[no head, from=5-5, to=4-3]
        	\arrow[no head, from=5-5, to=4-4]
        	\arrow[no head, from=5-5, to=4-6]
        	\arrow[no head, from=5-5, to=4-7]
        	\arrow[no head, from=4-3, to=3-4]
        	\arrow[no head, from=4-4, to=3-4]
        	\arrow[no head, from=4-6, to=3-4]
        	\arrow[no head, from=4-7, to=3-4]
        	\arrow[no head, from=4-3, to=3-6]
        	\arrow[no head, from=4-4, to=3-6]
        	\arrow[no head, from=4-6, to=3-6]
        	\arrow[no head, from=4-7, to=3-6]
        	\arrow[no head, from=3-4, to=2-3]
        	\arrow[no head, from=3-4, to=2-5]
        	\arrow[no head, from=3-4, to=2-7]
        	\arrow[no head, from=3-6, to=2-3]
        	\arrow[no head, from=3-6, to=2-5]
        	\arrow[no head, from=3-6, to=2-7]
        	\arrow[from=5-8, to=5-11]
        	\arrow[no head, from=5-16, to=4-14]
        	\arrow[no head, from=5-16, to=4-18]
        	\arrow[no head, from=8-16, to=7-14]
        	\arrow[no head, from=8-16, to=7-18]
        	\arrow[no head, from=6-15, to=7-14]
        	\arrow[no head, from=6-15, to=7-18]
        	\arrow[no head, from=6-17, to=7-14]
        	\arrow[no head, from=6-17, to=7-18]
        	\arrow[no head, from=5-16, to=6-15]
        	\arrow[no head, from=5-16, to=6-17]
        	\arrow[no head, from=4-14, to=3-15]
        	\arrow[no head, from=4-15, to=3-15]
        	\arrow[no head, from=4-17, to=3-15]
        	\arrow[no head, from=4-18, to=3-15]
        	\arrow[no head, from=4-14, to=3-17]
        	\arrow[no head, from=4-15, to=3-17]
        	\arrow[no head, from=4-17, to=3-17]
        	\arrow[no head, from=4-18, to=3-17]
        	\arrow[no head, from=5-16, to=4-15]
        	\arrow[no head, from=5-16, to=4-17]
        	\arrow[no head, from=3-15, to=2-14]
        	\arrow[no head, from=2-14, to=1-16]
        	\arrow[no head, from=2-16, to=1-16]
        	\arrow[no head, from=2-18, to=1-16]
        	\arrow[no head, from=3-17, to=2-18]
        	\arrow[no head, from=3-15, to=2-16]
        	\arrow[no head, from=3-17, to=2-16]
        	\arrow[no head, from=3-17, to=2-14]
        	\arrow[no head, from=3-15, to=2-18]
        	\arrow[no head, from=8-16, to=7-17]
        	\arrow[no head, from=8-16, to=7-15]
        	\arrow[no head, from=7-15, to=6-15]
        	\arrow[no head, from=7-15, to=6-17]
        	\arrow[no head, from=7-17, to=6-17]
        	\arrow[no head, from=7-17, to=6-15]
        	\arrow[no head, from=7-16, to=6-15]
        	\arrow[no head, from=7-16, to=6-17]
        \end{tikzcd}
    }
    \caption{
        Example of construction in Case \ref{Case-Deg-Height-1-Max-Diff}. For color-coding, see figure \ref{ex:running}.
    }
\end{figure}
\subsection{Injectivity of the Map $\psi$}\label{sec:injective}
We now consider a function $\psi$ which maps faces of $\OP_{C,O}(\tau)$ to faces of $\OP_{C',O'}(\tau)$.
Let $G'$ be a face of $\OP_{C',O'}(\tau)$ with codimension $t$, and let $G$ be a face of $\OP_{C,O}(\tau)$ such that $\psi(G)=G'$.
To show that $\psi$ is injective, we need to prove that $G$ is unique.

Let $\pi'$ be the face partition of $G'$, and let $\pi$ be the face partition of $G$. 
Then, we can immediately reconstruct all blocks of $G$ contained in $O'$. 
Next, we distinguish the set of variables set to $0$.
\subsubsection{The case $G'_{0,k+1}=\emptyset$}
There are two subcases.
\paragraph{Assume $G'_{\mathrm{eq},k+2}=\{y_{r'}^{k+2}\}$ or $G'_{\mathrm{eq},k+2}=\emptyset$.} The condition $G'_{\mathrm{eq},k+2}=\emptyset$ also includes the case when no chain inequality is tight.
Here, $y_{r'}^{{k+2}}$ has the minimal index $r'$ while in a singleton block in $\pi$.
Then, $G$' came from the Cases~\ref{Case-Gen-Single} or~\ref{Case-Deg-Single}.
Thus, we can reconstruct $\pi$, the sets $G_{0,i}$ for $1\leq i\leq k$, and if we use chains, we can fully reconstruct them as well.

\paragraph{Assume $G'_{\mathrm{eq},k+2}\neq\emptyset$ and $G'_{\mathrm{eq},k+2}\neq\{y_{r'}^{k+2}\}$.}
Then we are in the case \ref{Case-Deg-Height-1-Max-Diff}. Hence, we know that $G$ uses no chain and has a block $K\subseteq Y^{{k+1}}\cup Y^{{k+2}}$. 
We can reconstruct $K$ via 
\begin{equation*}
  K=G'_{\mathrm{eq},k+1}\cup G'_{\mathrm{eq},k+2}
\end{equation*}
and thus the whole partition $\pi$. 
Furthermore, we also have that we can reconstruct $G_{0,i}$ for $1\leq i\leq k$.

\subsubsection{The case $G'_{0,k+1}\neq\emptyset$}
We have $\{a_1,\dots,a_l\}\subseteq Y^{{k+1}}$ contained in a block $K\in\pi$ containing elements greater than them.
There are three subcases.


\paragraph{ All elements of $Y^{{k+2}}$ are contained in the same block in $\pi'$}
Then we know that $K$ has height at least $2$. Hence, we are in the Cases \ref{Case-Gen-Non-Single} or \ref{Case-Deg-Height-2}.
In this instance, $G'_{0,k+1}$ provides us exactly with the minimal elements of $K$. Thus, we can reconstruct $K$ via
\begin{equation*}
  K=T\cup G'_{0,k+1} \enspace,
\end{equation*}
where $T\in\pi'$ is the block containing $Y^{k+2}$.  In this way, we can recover the entire partition $\pi$.
Similarly, we can reconstruct all chain elements and also all variables set to $0$. 

\paragraph{Not all elements of $Y^{{k+2}}$ are contained in the same block in $\pi'$ and $G'$ uses a chain.}
Then we know that $K$ is of height exactly $1$, $G$ already uses a chain and we are in the Case \ref{Case-Gen-Non-Single}.
In this instance, we can reconstruct
\begin{equation*}
  K=G'_{0,k+1}\cup G'_{\mathrm{eq},k+2}
\end{equation*}
and thus, again, the whole partition $\pi$. 
We recover the chains and the remaining variables set to $0$.
\paragraph{Not all elements of $Y^{{k+2}}$ are contained in the same block in $\pi'$ and $G'$ uses no chain or $k=\ell-1$}
This implies we are in the Case \ref{Case-Deg-Height-1-Max-Standard} and hence we have that $K$ is of height exactly $1$ and $G$ uses no chain or $k=\ell-1$.
Then we get
\begin{equation*}
  K=G'_{0,k+1}\cup\{y_{r'}^{{k+2}}\} \enspace ,
\end{equation*}
where $r'$ is the minimal index such that $y_{r'}^{{k+2}}$ is a singleton in $\pi'$, or we have
\begin{equation*}
  K=G'_{0,k+1}\cup\{\hat{1}\}
\end{equation*}
if $k=\ell-1$.
This gives us the entire partition $\pi$, and so we can reconstruct all variables set to $0$.

\section{Computations}\label{sec-comp}

Here we discuss methods of computing the face numbers of the order and chain polytopes of an arbitrary finite poset and their implementation.
Order and chain polytopes of posets are unique in that both their vertex and facet descriptions can be derived from the poset through combinatorial means.
Such a pair of a point set and a set of linear inequalities is also known as the \emph{double description} of a polytope.
This means that convex hull computations, which can be tedious or even forbiddingly expensive, are unnecessary; see \cite{HDCG3:convex+hulls} for a survey and \cite{AvisBremnerSeidel:1997,polymake:2017} for computational experiments.
The incidences between vertices and facets are directly available from a double description.
Since the face lattice of a polytope is both atomic and coatomic \cite[\S2.2]{Ziegler:1995}, those vertex-facet incidences determine the entire face lattice.
Kaibel and Pfetsch found an output-sensitive algorithm for computing the entire Hasse diagram of the face lattice from the vertex-facet incidences of a polytope \cite{KaibelPfetsch:2002}.
This algorithm is implemented in \polymake \cite{polymake:2017}, which we used for our experiments.
Kliem and Stump recently described a slightly faster method to compute the faces without the covering relations, i.e., the nodes of the Hasse diagram but not the edges \cite{KliemStump:2022}.
Yet, for our purposes, we found the Kaibel--Pfetsch algorithm to be  sufficiently fast.

This leaves the question of how to obtain the vertices and the facets of the order and chain polytopes in the first place.
Each step in the procedure are standard graph algorithms.
The initialization phase is the same for the order and the chain polytope.
The input consists of the Hasse diagram of the extended poset $\hat{P}$, represented as a directed graph with arcs oriented upward.
Starting at the bottom element $\hat{0}$, we perform a depth-first search to obtain all maximal chains.
This gives the comparability graph $\comp(P)$, which is undirected, by enumerating all pairs of nodes in each maximal chain.
The most interesting subtask is to enumerate the maximal antichains.
These are precisely the maximal independent sets of $\comp(P)$, which in turn is the same as the set of maximal cliques in the complement of $\comp(P)$.
Computing maximal independent sets or, equivalently, maximal cliques is a key algorithmic problem in combinatorial optimization.
The decision problem CLIQUE, which asks if a given graph has a maximal clique (or independent set) of size at least $k$, is known to NP-complete \cite[GT19]{GareyJohnson:1979}.
This means that there is no hope for a fast algorithm.
Makino and Uno gave an algorithm for CLIQUE which performs well in practice \cite{MakinoUno:2004}, and this is implemented in \polymake.
Equipped with this information, formulating double descriptions of $\OP(P)$ and $\CP(P)$ becomes straightforward.
The details are spelled out in Algorithms~\ref{algo:order} and~\ref{algo:chain}.

\begin{figure}
\caption{Face numbers of the order and chain polytopes of $P_\tau$ for $n=\sum \tau_i = 10$, omitting trivial cases\\}
\scalebox{.8}{
  \rotatebox{90}{%
\parbox{1.7\textwidth}{%
\begin{tabularx}{\linewidth}{@{\extracolsep{\fill}}lrrrrrrrrrrcrrrrrrrrrr@{}}
\toprule
\multicolumn{1}{c}{poset} & 
\multicolumn{10}{c}{$f$-vector of $\OP(\tau)$} && \multicolumn{10}{c}{$f$-vector of $\mathcal{C}(\tau)$} \\
\midrule
$\tau$ &0&1&2&3&4&5&6&7&8&9&\hphantom{x}&0&1&2&3&4&5&6&7&8&9\\
\midrule
(2,2,1,1,1,1,1,1) &  13&74&245&526&770&784&554&265&81&14 && 13&74&245&526&770&784&554&265&81&14\\
(2,2,2,1,1,1,1) &  14&85&297&665&1002&1035&730&342&100&16 && 14&85&298&673&1029&1085&785&378&113&18\\
(2,2,2,2,1,1) &  15&97&358&838&1304&1371&967&443&123&18 && 15&97&361&863&1392&1541&1162&576&173&26\\
(2,2,2,2,2) &  16&110&429&1052&1695&1817&1281&572&150&20 && 16&110&435&1105&1893&2222&1770&920&285&42\\
(3,2,1,1,1,1,1) &  16&102&359&792&1162&1162&792&359&102&16 && 16&102&359&792&1162&1162&792&359&102&16\\
(3,2,2,1,1,1) &  17&116&433&998&1504&1518&1026&453&123&18 && 17&116&437&1028&1598&1678&1186&547&153&22\\
(3,2,2,2,1) &  18&131&519&1257&1964&2021&1364&586&150&20 && 18&131&528&1329&2205&2459&1831&878&249&34\\
(3,3,1,1,1,1) &  19&139&533&1230&1830&1810&1194&513&135&19 && 19&139&533&1230&1830&1810&1194&513&135&19\\
(3,3,2,1,1) &  20&156&641&1576&2466&2518&1674&703&174&22 && 20&156&645&1606&2564&2696&1866&825&216&28\\
(3,3,2,2) &  21&174&761&1972&3196&3310&2182&887&207&24 && 21&174&773&2078&3578&4036&2970&1377&369&46\\
(3,3,3,1) &  23&205&933&2430&3866&3878&2462&965&219&25 && 23&205&949&2542&4206&4446&3018&1281&315&37\\
(4,2,1,1,1,1) &  23&163&589&1285&1824&1739&1118&474&125&18 && 23&163&589&1285&1824&1739&1118&474&125&18\\
(4,2,2,1,1) &  24&184&710&1620&2354&2252&1427&587&148&20 && 24&184&721&1697&2577&2600&1744&756&197&26\\
(4,2,2,2) &  25&206&850&2048&3091&3011&1898&756&179&22 && 25&206&873&2221&3630&3914&2778&1256&333&42\\
(4,3,1,1,1) &  26&221&890&2051&2963&2797&1741&700&171&22 && 26&221&890&2051&2963&2797&1741&700&171&22\\
(4,3,2,1) &  27&245&1066&2632&4002&3885&2420&944&216&25 && 27&245&1077&2709&4236&4277&2806&1166&285&34\\
(4,3,3) &  30&315&1548&4069&6263&5927&3503&1272&267&28 && 30&315&1592&4355&7067&7159&4599&1836&423&46\\
(4,4,1,1) &  33&352&1513&3485&4878&4388&2578&972&221&26 && 33&352&1513&3485&4878&4388&2578&972&221&26\\
(4,4,2) &  34&383&1813&4565&6852&6445&3835&1406&295&30 && 34&383&1824&4642&7086&6848&4254&1664&381&42\\
(5,2,1,1,1) &  38&285&1015&2125&2856&2559&1540&610&150&20 && 38&285&1015&2125&2856&2559&1540&610&150&20\\
(5,2,2,1) &  39&321&1228&2687&3680&3289&1941&744&175&22 && 39&321&1254&2856&4130&3931&2475&1005&245&30\\
(5,3,1,1) &  41&388&1562&3463&4733&4195&2445&920&210&25 && 41&388&1562&3463&4733&4195&2445&920&210&25\\
(5,3,2) &  42&427&1875&4472&6422&5820&3371&1224&261&28 && 42&427&1901&4641&6898&6553&4030&1570&360&40\\
(5,4,1) &  48&624&2694&5943&7841&6616&3645&1290&275&30 && 48&624&2694&5943&7841&6616&3645&1290&275&30\\
(6,2,1,1) &  69&520&1774&3502&4408&3690&2075&769&177&22 && 69&520&1774&3502&4408&3690&2075&769&177&22\\
(6,2,2) &  70&587&2160&4446&5672&4710&2587&926&204&24 && 70&587&2217&4788&6504&5792&3411&1298&297&34\\
(6,3,1) &  72&716&2778&5795&7396&6116&3333&1176&252&28 && 72&716&2778&5795&7396&6116&3333&1176&252&28\\
(7,2,1) &  132&964&3097&5708&6692&5222&2744&953&206&24 && 132&964&3097&5708&6692&5222&2744&953&206&24\\
\bottomrule
\end{tabularx}
    \label{tab:face-numbers}
    }%
  }%
  }
\end{figure}

\begin{algorithm}[th]\DontPrintSemicolon
  \caption{Order polytope}
  \label{algo:order}
  \KwIn{Hasse diagram of extended poset $\hat{P}$}
  \KwOut{Double description of $\OP(P)$}

  compute maximal chains of $P$ \;
  compute comparability graph $\comp(P)$ \;
  compute maximal antichains = maximal independent sets in $\comp(P)$ \;
  $V$ $\leftarrow$ $\emptyset$, $F$ $\leftarrow$ $\emptyset$ \;
  \ForEach{maximal antichain $A$}{
    \ForEach{subset $S\subseteq A$}{
      compute order filter $O$ generated by $S$ \;
      add characteristic vector $\chi_O$ to $V$ \;
    }
  }
  \ForEach{arc $(p,q)$ of Hasse diagram}{
    add inequality $x_p \leq x_q$ to $F$ \;
  }
  \Return $(V,F)$\;
\end{algorithm}

\begin{algorithm}[th]\DontPrintSemicolon
  \caption{Chain polytope}
  \label{algo:chain}
  \KwIn{Hasse diagram of extended poset $\hat{P}$}
  \KwOut{Double description of $\CP(P)$}
  compute maximal chains of $P$ \;
  compute comparability graph $\comp(P)$ \;
  compute maximal antichains = maximal independent sets in $\comp(P)$ \;
  $V$ $\leftarrow$ $\emptyset$, $F$ $\leftarrow$ $\emptyset$ \;
  \ForEach{maximal antichain $A$}{
    \ForEach{subset $S\subseteq A$}{
      add characteristic vector $\chi_S$ to $V$ \;
    }
  }
  \ForEach{maximal chain $p_1 \prec p_2 \prec \dots \prec p_s$}{
    add inequality $x_{p_1} + \dots + x_{p_s} \leq 1$ to $F$ \;
  }
  \Return $(V,F)$ \;
\end{algorithm}

Of course, the maximal ranked posets $P_\tau$ allow for considerable simplifications of the above methods.
In particular, the maximal antichains are precisely the sets of elements of $p$ of fixed rank.
\bigskip

The \polymake implementation of Algorithms \ref{algo:order} and \ref{algo:chain} is available since version 4.11.
With this computing the entire Table~\ref{tab:face-numbers} takes less than a minute on a standard laptop.
Within about 45 minutes we can also find (see Table~\hyperlink{tab:f-vector-running}{5}) the $f$-vectors of our running example $\tau=(5,2,1,4,2,3)$ from Figure~\ref{ex:running}.

\begin{figure}
\hypertarget{tab:f-vector-running}{}
\footnotesize
\caption{$f$-vector of the running example $\tau=(5,2,1,4,2,3)$}
\noindent
\begin{tabularx}{\linewidth}{@{\extracolsep{\fill}}lccccccccc@{}}
  \toprule
   &0&1&2&3&4&5&6&7&8 \\
  \midrule
  $\OP(\tau)$ & 61& 1306& 13459& 79115& 296362& 759353& 1393462& 1887296& 1922781\\
  $\CP(\tau)$ &61 &1306 &13935 &87979 &364142 &1053486 &2220180 &3500405 &4196664 \\
  \midrule
  &9&10&11&12&13&14&15&16 \\
  \midrule
  $\OP(\tau)$ & 1488969 & 878903& 393545& 131842& 32207& 5492& 607& 38 \\
  $\CP(\tau)$ &3857441 &2720641& 1462271 &589116& 172550& 34780& 4336& 257\\
  \bottomrule
\end{tabularx}

\end{figure}

\printbibliography
\end{document}